\documentclass[final,a4paper,12pt,]{article}
\usepackage{a4wide}
\date{\today}

\usepackage[T1]{fontenc}
\usepackage{ae,aecompl}
\usepackage[final,tracking=true,kerning=true,spacing=true,stretch=10,shrink=10]{microtype}
\SetTracking{encoding=*, shape=sc}{50}
\microtypecontext{spacing=nonfrench}
\usepackage[utf8]{inputenc}
\usepackage[british]{babel}
\usepackage{csquotes}
\usepackage[draft=false]{hyperref}
\usepackage{xcolor}
\usepackage{graphicx}
\usepackage{amsfonts}
\usepackage{amssymb}
\usepackage{amsmath}
\usepackage{amsthm}
\usepackage{tikz-cd}
\usepackage{float}
\usepackage{graphicx}
\usepackage{tikz}
\usepackage{thmtools}
\usepackage{thm-restate}
\usepackage{mathrsfs}
\usepackage{enumitem}
\usepackage{mathtools}
\setenumerate{label=(\arabic*),topsep=.1\baselineskip,itemsep=-0.1\baselineskip}
\setitemize{topsep=.1\baselineskip,itemsep=-0.1\baselineskip}
\usepackage[backend=biber,
url=false,
isbn=false,
backref=true,
citestyle=alphabetic,
bibstyle=alphabetic,
autocite=inline,
maxnames=99,
minalphanames=4,
maxalphanames=4,
sorting=nyt,]{biblatex}
\hypersetup{
	colorlinks,
	linkcolor={red!50!black},
	citecolor={blue!50!black},
	urlcolor={blue!80!black}
}

\newcommand{\fC}{{\mathfrak C}}
\newcommand{\cC}{{\mathcal C}}

\newcommand{\bN}{{\mathbf{N}}}
\newcommand{\bR}{{\mathbf{R}}}
\newcommand{\bZ}{{\mathbf{Z}}}

\newcommand\Lasc{{\mathrm{L}}}

\newcommand{\equivLasc}{\mathrel{\equiv_{\Lasc}}}

\newcommand{\restr}{\mathord{\upharpoonright}}

\newcommand{\biglor}{\bigvee}

\newcommand{\eq}{\mathrm{eq}}
\newcommand{\Sk}{\mathrm{Sk}}

\DeclareMathOperator{\tp}{{tp}}
\DeclareMathOperator{\acl}{{acl}}
\DeclareMathOperator{\dcl}{{dcl}}
\DeclareMathOperator{\Th}{{Th}}
\DeclareMathOperator{\Gal}{{Gal}}

\DeclareMathOperator{\id}{{id}}
\DeclareMathOperator{\Aut}{{Aut}}

\DeclareMathOperator{\SL}{{SL}}

\DeclareMathOperator{\Lascd}{{d_\Lasc}}

\newcommand{\forkindep}[1][]{%
	\mathrel{
		\mathop{
			\vcenter{
				\hbox{\oalign{\noalign{\kern-.3ex}\hfil$\vert$\hfil\cr
						\noalign{\kern-.7ex}
						$\smile$\cr\noalign{\kern-.3ex}}}
			}
		}\displaylimits_{#1}
	}
}

\newtheorem*{mainthm*}{Main Theorem}

\newtheorem{thm}{Theorem}[section]

\newtheorem{lem}[thm]{Lemma}
\newtheorem{fct}[thm]{Fact}
\newtheorem{cor}[thm]{Corollary}
\newtheorem{prop}[thm]{Proposition}
\newtheorem{qu}[thm]{Question}
\newtheorem{con}[thm]{Conjecture}

\theoremstyle{remark}
\newtheorem{rem}[thm]{Remark}
\theoremstyle{definition}
\newtheorem{dfn}[thm]{Definition}

\newtheorem*{clm*}{Claim}
\newtheorem{ex}[thm]{Example}
\newcounter{claimcounter}[thm]

\newenvironment{clmproof}[1][\proofname]{\proof[#1]}{\endproof}

\newcommand{\xqed}[1]{%
	\leavevmode\unskip\penalty9999 \hbox{}\nobreak\hfill
	\quad\hbox{\ensuremath{#1}}}

\makeatletter
\renewcommand*\author[1]{%
	\stepcounter{author}%
	\ifnum\c@author=1
	\gdef\@author{#1}%
	\else
	\xdef\@author{\unexpanded\expandafter{\@author\and#1}}%
	\fi
	\csgdef{author@\the\c@author}{#1}}
\newcommand*\grant[1]{%
	\xdef\@author{\unexpanded\expandafter{\@author\footnote{#1}}}
}
\newcommand*\email[1]{%
	\csgdef{email@\the\c@author}{#1}}
\newcommand*\orcid[1]{%
	\csgdef{orcid@\the\c@author}{#1}}
\newcommand*\address[1]{%
	\csgdef{address@\the\c@author}{#1}}
\AtEndDocument{%
	\xdef\author@count{\the\c@author}%
	\c@author=1
	\print@authors}
\newcommand*\print@authors{%
	\ifnum\c@author>\author@count
	\else
	\print@author{\the\c@author}%
	\advance\c@author by 1
	\expandafter\print@authors
	\fi}
\newcommand*\print@author[1]{%
	\par\medskip
	\begin{tabular}{@{}l@{}}%
		\textsc{\csuse{author@#1}}\\
		\csuse{address@#1}\\
		\href{\csuse{orcid@#1}}{\includegraphics[height=\fontcharht\font`\B]{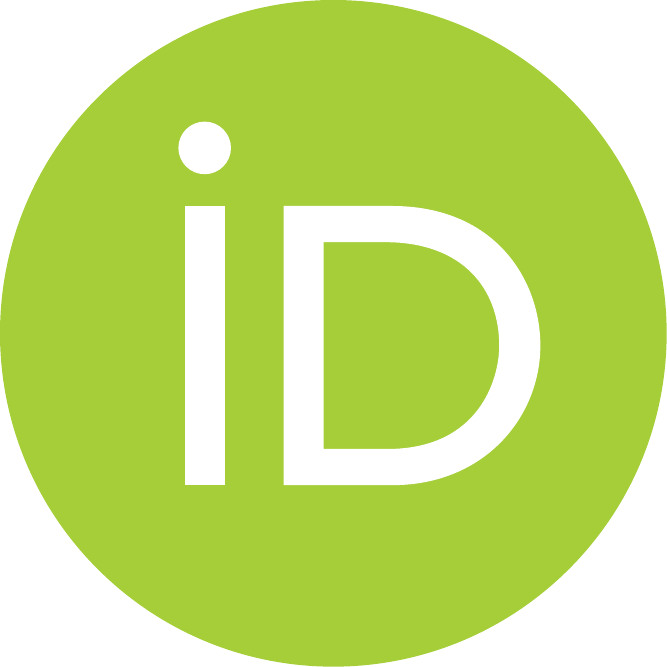} \csuse{orcid@#1}}\\
		\href{mailto:\csuse{email@#1}}{\csuse{email@#1}}
\end{tabular}}
\makeatother
\makeatletter
\newcommand{\subjclass}[2][1991]{%
	\let\@oldtitle\@title%
	\gdef\@title{\@oldtitle\footnotetext{#1 \emph{Mathematics subject classification.} #2}}%
}
\newcommand{\keywords}[1]{%
	\let\@@oldtitle\@title%
	\gdef\@title{\@@oldtitle\footnotetext{\emph{Key words and phrases.} #1.}}%
}
\makeatother

\title{Hereditary G-compactness}

\author{Tomasz Rzepecki}
\grant{The author was supported by the Narodowe Centrum Nauki grant no. 2016/22/E/ST1/00450, the doctoral scholarship 2017/24/T/ST1/00224 and the Lady Davis fellowship}
\address{The Hebrew University of Jerusalem and Uniwersytet Wrocławski}
\orcid{https://orcid.org/0000-0001-9786-1648}
\email{tomasz.rzepecki@math.uni.wroc.pl}

\usepackage{isomath}

\usepackage{upgreek}
\usepackage{wrapfig}

\begin{document}
	\maketitle

	\begin{abstract}
		We introduce the notion of hereditary G-compactness (with respect to interpretation). We provide a sufficient condition for a poset to not be hereditarily G-compact, which we use to show that any linear order is not hereditarily G-compact. Assuming that a long-standing conjecture about unstable NIP theories holds, this implies that an NIP theory is hereditarily G-compact if and only if it is stable (and by a result of Simon, this holds unconditionally for $\aleph_0$-categorical theories). We show that if $G$ is definable over $A$ in a hereditarily G-compact theory, then $G^{00}_A=G^{000}_A$.
		We also include a brief survey of sufficient conditions for G-compactness, with particular focus on those which can be used to prove or disprove hereditary G-compactness for some (classes of) theories.
	\end{abstract}
	
	The notion of G-compactness (see Definition~\ref{dfn:g_comp}) was originally introduced by Lascar in his paper \cite{Las82}. It is related to the so-called Lascar strong types and Galois groups of first order theories, which are important objects of study in contemporary model theory.
	
	G-compactness a generalisation of stability and simplicity (see Fact~\ref{fct:stable_simple_gc}), so it is a tameness-like property of a first-order theory. Unfortunately, unlike these two properties, it is not preserved by interpretation, and not even by reducts and adding constants.
	
	In this paper, we introduce a stronger property of \emph{hereditary G-compactness} (Definition~\ref{dfn:hered_g_comp}) which is (by definition) preserved under interpretation, and is thus, in some ways, much more well-behaved as a model-theoretic property.
	
	An interesting consequence of hereditary G-compactness is that for a group $G$ definable in a hereditary G-compact theory, we have, for every small $A$, $G^{00}_A=G^{000}_A$ (see Proposition~\ref{prop:def_group}). It seems likely that it could also imply other similar type-definability results.
	
	The main result is the following.
	
	\begin{mainthm*}[Theorem~\ref{thm:lin_not_hgc}]
		Let $T$ be any theory such that $T$ interprets an infinite linear order. Then $T$ is not hereditarily G-compact.\xqed{\lozenge}
	\end{mainthm*}
	
	In order to prove the main theorem, we generalise \cite{CLPZ01} (where the authors gave the first example of a non-G-compact theory) and use this generalisation to describe (via Theorem~\ref{thm:wts_not_hgc}) a wide class of posets (including all infinite linear orders) whose theories are not hereditarily G-compact (see Theorem~\ref{thm:lin_not_hgc}).
	
	The Main~Theorem implies that, modulo a long-standing Conjecture~\ref{con:unstable_nip_linear}, an NIP theory is stable if and only if it is hereditarily G-compact. In the other direction, it seems plausible that all NSOP$_1$ theories are hereditarily G-compact (see Fact~\ref{fct:gcomp_nsop1} and the surrounding discussion), closely tied to the NSOP (cf.\ Question~\ref{qu:SOP_nhgc}).
	
	Unfortunately, obtaining positive results (namely, proving hereditary G-compactness) seems to be rather difficult, and typically involves some deep results. There seems to be no obvious way to do this even in specific examples, but at least a partial result seems possible in the case of $\aleph_0$-categorical structures with q.e.\ in a finite relational language (cf.\ Question~\ref{qu:henson} and the preceding discussion). Likewise, proving lack of (hereditary) G-compactness for even particular NSOP theories seems to be a hard problem, since essentially all understood examples of NSOP theories appear to be G-compact for very general reasons (see Example~\ref{ex:canonical_jep} and Remark~\ref{rem:existence_for_nsop1}).
	
	The outline of the paper is as follows. In Section~\ref{sec:classical}, we briefly recall the classical notions of Lascar distance and G-compactness and the basic relevant facts concerning them. In Section~\ref{sec:hgc}, we introduce the notion of hereditary G-compactness, provide some examples and discuss the relationship with the connected group components. In Section~\ref{sec:thrsp}, we introduce some technical notions necessary to prove the main theorem. In Section~\ref{sec:main_proof} we prove the main theorem. In Section~\ref{sec:criteria} we survey some of the known sufficient conditions for G-compactness, and how (and whether) they might be used to prove hereditary G-compactness. Finally, in Section~\ref{sec:open}, we list some open problems.
	
	\section{Lascar distance, G-compactness}
	\label{sec:classical}
	
	\begin{dfn}
		Given a monster model $\fC$ and two (possibly infinite) small tuples $a,b\in \fC$, we say that $\Lascd(a,b)\leq n$ if there are sequences $a=a_0,a_1,\ldots,a_n=b$ and $M_1,\ldots, M_n$ such that for $i=1,\ldots,n$ we have $M_i\preceq \fC$ and $a_{i-1}\equiv_{M_i} a_i$.
		
		The \emph{Lascar distance} $\Lascd$ between $a$ and $b$ is the smallest natural number $n$ such that $\Lascd(a,b)\leq n$, or $\infty$ if no such $n$ exists.
		
		We say that $a,b$ are \emph{Lascar equivalent} or have the same \emph{Lascar strong type}, written $a\equivLasc b$, if $\Lascd(a,b)<\infty$.
		
		The \emph{Lascar strong type of $a$} is its $\equivLasc$-class.
		\xqed{\lozenge}
	\end{dfn}
	
	\begin{rem}[Lascar graph]
		\label{rem:Lascar_graph}
		Another way to describe the Lascar distance $\Lascd$ is to say that it is the distance in the undirected graph $(V,E)$ (called the Lascar graph), where $V$ is the set of all small tuples in $\fC$ and $E$ is the set of all pairs $a,b$ such that for some $M\preceq \fC$ we have $\tp(a/M)=\tp(b/M)$.
		
		The Lascar graph (and distance) is also often defined in terms of indiscernible sequence. Namely, we declare that an edge exists between $a$ and $b$ if $(a,b)$ can be extended to an infinite indiscernible sequence. The resulting distance function is bi-Lipschitz equivalent to $\Lascd$ given above (so the corresponding $\equivLasc$ is the same).\xqed{\lozenge}
	\end{rem}

	
	
	\begin{dfn}
		\label{dfn:g_comp}
		We say that a theory $T$ is G-compact if every Lascar strong type has finite diameter, i.e.\ for every tuple $a$ we have an integer $n$ such that $\Lascd(a,b)<\infty$ implies that $\Lascd(a,b)\leq n$.
		
		We say that a theory $T$ is \emph{$n$-G-compact} if for any tuples $a,b$ we have that if $\Lascd(a,b)<\infty$, then $\Lascd(a,b)\leq n$. (Or equivalently, if $\Lascd(a,b)\leq n+1$ implies $\Lascd(a,b)\leq n$.)\xqed{\lozenge}
	\end{dfn}

	\begin{rem}
		The relation $\equivLasc$ has many equivalent definitions. Among others, it is the finest bounded invariant equivalence relation. However, in this paper, we will only really use the definition provided above.\xqed{\lozenge}
	\end{rem}
	
	The canonical example of a non-G-compact theory has been described in \cite{CLPZ01}. It consists of a structure with infinitely many disjoint sorts $M_n$, with $M_n$ being $n$-G-compact, but not $(n-1)$-G-compact (which is enough for non-G-compactness by Fact~\ref{fct:gcom_equivalent}). We will imitate this construction to prove the main theorem.
	
	\begin{prop}
		\label{prop:ext1}
		If $a\equivLasc a'$ and $b$ is arbitrary, then there is some $b'$ such that $ab\equivLasc a'b'$.
	\end{prop}
	\begin{proof}
		By definition, there is a finite sequence $M_1,\ldots,M_n$ of models and automorphisms $\sigma_i\in \Aut(\fC/M_i)$ such that $a'=\sigma_n\ldots\sigma_1(a)$. Then $b'=\sigma_n\ldots\sigma_1$ is as described.
	\end{proof}
	
	\begin{prop}
		\label{prop:ext2}
		If $a,a'$ and $b,b'$ are pairs of tuples of the same length, then $\Lascd(a,a')\leq \Lascd(ab,a'b')$.
	\end{prop}
	\begin{proof}
		If $\Lascd(ab,a'b')\leq n<\infty$, this is witnessed by a sequence of $n$ models and a sequence of automorphisms fixing the respective models. The same sequence witnesses that $\Lascd(a,a')\leq n$.
	\end{proof}
	
	The following fact is folklore.
	\begin{fct}
		\label{fct:gcom_equivalent}
		The following are equivalent.
		\begin{enumerate}
			\item 
			$T$ is G-compact.
			\item 
			For some $n$, one of the following (equivalent) conditions holds:
			\begin{itemize}
				\item
				$T$ is $n$-G-compact for some $n$.
				\item
				For all \emph{finite} tuples $a,b$, if $\Lascd(a,b)\leq n+1$, then $\Lascd(a,b)\leq n$.
			\end{itemize}
		\end{enumerate}
	\end{fct}
	\begin{proof}
		It is clear that the first bullet in (2) implies the second bullet. For the converse, we first use compactness to deduce that the impliation holds also for infinite tuples. Then we observe that by induction with respect to $N$, if $\Lascd(a,b)\leq N$, then $\Lascd(a,b)\leq n$.
		
		It is also clear that (2) implies (1).
		
		To see that (1) implies (2), we argue by contraposition and use a diagonal argument. More precisely suppose $(a_n)_{n\in \bN},(b_n)_{n\in \bN}$ are tuples such that for each $n$, $\Lascd(a_n,b_n)>n$ and $a_n\equivLasc b_n$. By Proposition~\ref{prop:ext1}, we can choose for each $m\in \bN$ a sequence $(b^m_n)_{n\in \bN}$ such that $b^m_m=b_m$ and $(a_n)_{n\in \bN}\equivLasc (b^m_n)_{n\in \bN}$. Then by Proposition~\ref{prop:ext2}, we have $\infty>\Lascd((a_n)_n,(b_n^m)_n)> m$. It follows that the diameter of the Lascar strong type of $(a_n)_n$ is infinite.
	\end{proof}
	
	\begin{rem}
		There are other equivalent characterisations of G-compactness we will not use in this paper, e.g.:
		\begin{itemize}
			\item
			the (Lascar) Galois group $\Gal(T)$ is Hausdorff (with the logic topology),
			\item
			for some tuple $m$ enumerating a small model, the class $[m]_{\equivLasc}$ has finite $\Lascd$-diameter,
			\item
			for some tuple $m$ enumerating a small model, the class $[m]_{\equivLasc}$ is type-definable. \xqed{\lozenge}
		\end{itemize}
	\end{rem}

%
%
%
%
%

	\section{Hereditary G-compactness}
	\label{sec:hgc}
	As we will see later in this section, G-compactness is not preserved under interpretations --- even adding or forgetting a single constant symbol can turn a G-compact theory in to a non-G-compact one and vice versa. Thus, to obtain a more well-behaved property of a theory, it seems natural to consider the following.
	\begin{dfn}
		\label{dfn:hered_g_comp}
		A theory $T$ is said to be \emph{hereditarily [$n$-]G-compact} if for every model $M\models T$, and every structure $N$ interpreted by $M$ (with parameters), $\Th(N)$ is [$n$-]G-compact.
		
		A theory $T$ is said to be \emph{weakly hereditarily [$n$-]G-compact} if for every model $M\models T$ and every reduct $N$ of $M$ (possibly after adding some parameters), $\Th(N)$ is [$n$-]G-compact.\xqed{\lozenge}
	\end{dfn}
	
	\begin{rem}
		It is not hard to see that for hereditary [$n$-]G-compactness of $T$, it is enough to consider reducts of models of $T^\eq$ (so hereditary [$n$-]G-compactness of $T$ is equivalent to weakly hereditary [$n$-]G-compactness of $T^\eq$): we can simply forget all the irrelevant structure, resulting in a collection of sorts with no structure, which will not affect [$n$-]G-compactness in any way.
		
		Likewise, it is clear that for weakly hereditary [$n$-]G-compactness of $T^\eq$, it is enough to consider the reducts of $T^\eq$ expanded by \emph{real} constants.\xqed{\lozenge}
	\end{rem}

	\subsection*{Examples}
	
	\begin{ex}
		\label{ex:skolem}
		If $T$ has definable Skolem functions, then any expansion of $T$ by constants is $1$-G-compact (but Example~\ref{ex:o-mingrp} shows that it is not hereditarily G-compact). This follows from the fact that in all of these expansions, $\dcl(\emptyset)$ is a model.
		
		It follows that if $T$ is an arbitrary non-G-compact theory, then $T^{\Sk}$, the Skolemization of $T$, is G-compact but not weakly hereditarily G-compact (because $T$ is a reduct of $T^{\Sk}$).\xqed{\lozenge}
	\end{ex}
	
	\begin{ex}
		Take any non-G-compact theory $T$ in a relational language $L$ such that $\dcl(\emptyset)\neq \emptyset$. Let $T_\infty$ be the $L$-theory of infinitely many disjoint models of $T$ (with the symbols of $L$ interpreted naturally within each model of $T$, but with no relations between models).
		
		Then one can show that $T_\infty$ is $1$-G-compact, but adding any parameter corresponding to an element of $\dcl(\emptyset)$ in a model of $T$ makes it not G-compact (this parameter makes each sort of this model definable and, in fact, stably embedded).\xqed{\lozenge}
	\end{ex}
	
	\begin{ex}
		\label{ex:simple_is_hgc}
		It is well-known that simplicity of a structure is preserved by interpretation. Since every simple theory is G-compact (Fact~\ref{fct:stable_simple_gc}), it follows that every simple theory is hereditarily $2$-G-compact. Similarly, every stable theory is hereditarily $1$-G-compact.\xqed{\lozenge}
	\end{ex}
	
	\begin{ex}
		\label{ex:o-mingrp}
		Any o-minimal expansion of a group (with a definable element distinct from the identity) is $1$-G-compact, because it has definable Skolem functions. More generally, Fact~\ref{fct:inpm_gcomp} implies that all o-minimal structures are $2$-G-compact. On the other hand, Theorem~\ref{thm:lin_not_hgc} implies that no o-minimal structure is hereditarily G-compact.
		\xqed{\lozenge}
	\end{ex}
	
	\begin{ex}
		Furthermore, NTP$_2$ with existence over $\emptyset$ and NSOP$_1$ theories with existence are $2$-G-compact (see Fact~\ref{fct:gc_in_ntp2} and Fact~\ref{fct:gcomp_nsop1}). This includes in particular simple theories and o-minimal theories. The o-minimal examples show that NTP$_2$ is not sufficient for hereditary G-compactness; the question about whether or not NSOP$_1$ is sufficient for hereditary G-compactness remains open (as it is not known whether all NSOP$_1$ theories have existence).\xqed{\lozenge}
	\end{ex}
	
	\subsection*{Hereditary G-compactness and connected group components}
	In this section, we will see some basic consequences of hereditary G-compactness of $T$ for the groups definable in $T$. Recall the notions of model-theoretic connected components of a definable group.
	\begin{dfn}
		Suppose $G$ is a group definable in $\fC$ with parameters in a small set $A$. Then $G^{000}_A$ is the smallest subgroup of $G=G(\fC)$ which is invariant under $\Aut(\fC/A)$ and which has small index in $G$ (i.e.\ no greater than $2^{\lvert T\rvert+\lvert A\rvert}$). Similarly, $G^{00}_A$ is the smallest subgroup of $G$ which is type-definable with parameters in $A$ and has small index.
		
		If $G^{000}_A$ does not depend on $A$ (over which $G$ is definable), then we write simply $G^{000}$ for $G^{000}_A$. Likewise, if $G^{00}_A$ does not depend on $A$, we write $G^{00}$. In these cases, we say that $G^{000}$ or $G^{00}$ (respectively) exists.\xqed{\lozenge}
	\end{dfn}
	(Note that clearly $G^{00}_A\geq G^{000}_A$.)
	
	\begin{fct}
		If $G$ is definable in $\fC\models T$ and $T$ has NIP, then $G^{00}$ and $G^{000}$ exist.
	\end{fct}
	\begin{proof}
		See \cite[Theorem 5.3, Remark 5.1]{Gis11}.
	\end{proof}

	In \cite{GiNe08}, the authors consider the following construction: starting with a structure $M$ and a group $G(M)$ definable in $M$, they construct a structure $N=(M,X,\cdot)$, where $M$ has its original structure, $\cdot\colon G(M)\times X\to X$ is a free and transitive action, and there is no other structure on $X$ ($X$ is an ``affine copy of $G$''). They analyse the resulting structure, showing in particular that $\Aut(N)=G(M)\rtimes \Aut(M)$, as well as the following fact.
	\begin{fct}
		\label{fct:from_GN}
		If $M\models T$ and $G(M)$ is definable in $M$ without parameters, then the theory of $N=(M,X,\cdot)$ described above is G-compact if and only if $T$ is G-compact and $G^{00}_\emptyset=G^{000}_\emptyset$ (where $G=G(\fC)$ for the monster model $\fC\succeq M$).
	\end{fct}
	\begin{proof}
		This is \cite[Corollary 3.6]{GiNe08}. Note that the authors use the notation $G^*_L$ and $G^{\infty}_\emptyset$ for $G^{000}_\emptyset$.
	\end{proof}
	
	\begin{rem}
		\label{rem:single_constant}
		If $T$ is G-compact, $M\models T$ and $G$ is definable in $M$ without parameters, with $G^{00}_\emptyset\neq G^{000}_\emptyset$, then the theory of $N=(M,X,\cdot)$ defined as in \cite{GiNe08} is not G-compact, but becomes G-compact as soon as we add a constant symbol for an element of $X$. Conversely, if we fix any $x_0\in X$, then $(M,X,\cdot)$ is a non-G-compact reduct of the G-compact $(M,X,\cdot,x_0)$.\xqed{\lozenge}
	\end{rem}
	
	\begin{prop}
		\label{prop:def_group}
		If $T$ is hereditarily G-compact, $A$ is a small set and $G$ is a group definable in $T$ over $A$, then we have $G^{00}_A=G^{000}_A$. In particular, if $T$ has NIP, then $G^{00}=G^{000}$.
	\end{prop}
	\begin{proof}
		Immediate by Fact~\ref{fct:from_GN}, as $M$ clearly interprets $N$.
%
	\end{proof}
	Note that G-compactness alone certainly does not guarantee $G^{00}=G^{000}$, not even under NIP. For instance, the group $G=\widetilde{\SL_2(\bR^*)}$ from \cite[Theorem 3.2]{CP12} is definable in $M=((\bR,+,\cdot),(\bZ,+))$, which is G-compact (even after adding some parameters), by o-minimality of the reals and by stability of the integers. The proof of \cite[Theorem 3.2]{CP12} shows that $G^{00}\neq G^{000}$.

	\section{Three-splitting and cyclic three-splitting}
	\label{sec:thrsp}
	In this section, we introduce several properties of posets which will be necessary to prove the main lemmas.
	\subsection*{Linear sum; three-splitting}
	Linear sum is an elementary operation on partially ordered sets.
	\begin{dfn}
		Given two posets $P=(P,<),(Q,<)$, the \emph{linear sum} $P\oplus Q$ is defined as $(P\sqcup Q,<)$ where $a<b$ if either $a\in P$ and $b\in Q$ or $a<b$ in one of $P$, $Q$.
		
		Likewise, given an integer $n$, $P^{\oplus n}$ is the linear sum of $n$ copies of $P$ (note that this is the same as $[n]\times P$, where $[n]=\{1,\ldots,n\}$ has the natural ordering).\xqed{\lozenge}
	\end{dfn}
	Informally speaking, $P\oplus Q$ is the disjoint union of $P$ and $Q$ with $Q$ put after (or above) $P$.
%

	\begin{dfn}
		We say that a partially ordered set $(P,<)$ is \emph{initially self-additive} if the initial embedding of $P$ into $P^{\oplus 2}$ is elementary.\xqed{\lozenge}
	\end{dfn}
	The following proposition shows that initial self-additivity has some rather strong model-theoretic consequences.
	\begin{prop}
		\label{prop:self_add_propagates},
		Suppose that $P$ is an initially self-additive poset. Then for every two posets $Q_1,Q_2\equiv P$, the initial embedding of $Q_1$ into $Q_1\oplus Q_2$ is elementary.
	\end{prop}
	\begin{proof}
		Note that it is easy to see that if $P_1,P_2,Q_1,Q_2$ are posets such that $P_1\equiv Q_1$ and $P_2\equiv Q_2$, then $P_1\oplus P_2\equiv Q_1\oplus Q_2$, and even $(P_1\oplus P_2,P_1,P_2)\equiv (Q_1\oplus Q_2,Q_1,Q_2)$ (as posets with additional predicates for $P_1$ and $P_2$ or $Q_1$ and $Q_2$, respectively).
		
		But $P_1$ being an elementary substructure of $P_1\oplus P_2$ is clearly an elementary property of $(P_1\oplus P_2,P_1,P_2)$. The proposition follows by taking $P_1=P_2=P$.
%
%
%
	\end{proof}

	
	\begin{rem}
		An initially self-additive poset can have no maximal elements and no finite maximal chains.\xqed{\lozenge}
	\end{rem}
	
	\begin{prop}
		\label{prop:no_endpoints_three_splitting}
		If $(P,<)$ is linear, has no endpoints and is dense or discrete, then it is initially self-additive.
	\end{prop}
	\begin{proof}
		If $P$ is dense, then the theory of $(P,<)$ is the theory of dense linear orderings, and it has quantifier elimination. In particular, $P$ and $P^{\oplus 2}$ have the same theory, and by q.e., the embedding is elementary.
		
		The theory of discrete linear orders without endpoints is complete and it defines the successor function $S$; furthermore, it has quantifier elimination in the language $(<,S)$. As in the dense case, if $P$ is discrete without endpoints, then so is $P^{\oplus 2}$. The initial embedding of $P$ in $P^{\oplus 2}$ is a substructure in the $(<,S)$ language, so it is elementary.
	\end{proof}

	\subsection*{Cyclic orders; cyclic three-splitting}
	
	\begin{dfn}
		A ternary relation $C(x,y,z)$ is a (strict) partial cyclic order on a set $G$ if it satisfies the following axioms:
		\begin{enumerate}
			\item
			cyclicity: if $C(x,y,z)$, then $C(z,x,y)$,
			\item 
			asymmetry: if $C(x,y,z)$, then $\neg C(z,y,x)$,
			\item 
			transitivity: if $C(x,y,z)$ and $C(y,z,t)$, then $C(x,y,t)$.\xqed{\lozenge}
		\end{enumerate}
	\end{dfn}

	\begin{rem}[dummy constants]
		\label{rem:dummy_constants}
%
		For any structure $M$ in which at least one sort has more than one element, and every positive integer $n$, there is a pointwise definable subset of $M^{\eq}$ with exactly $n$ elements.\xqed{\lozenge}
	\end{rem}
	
	\begin{dfn}
		\label{dfn:C_n}
		Given a poset $P$ and $n\in \bN\setminus \{0\}$, we define $C_n(P)$ as $(P\times [n],C,R_n)$, where $[n]$ is the set of integers $\{1,\ldots,n\}$, $C$ is the natural cyclic ordering (induced from  $P^{\oplus n}$), while $R_n$ is the automorphism given by $R_n(p,n)=(p,1)$ and $R_n(p,i):=(p,i+1)$ for $i<n$. \xqed{\lozenge}
	\end{dfn}
	
	\begin{figure}
		\includegraphics[width=\textwidth]{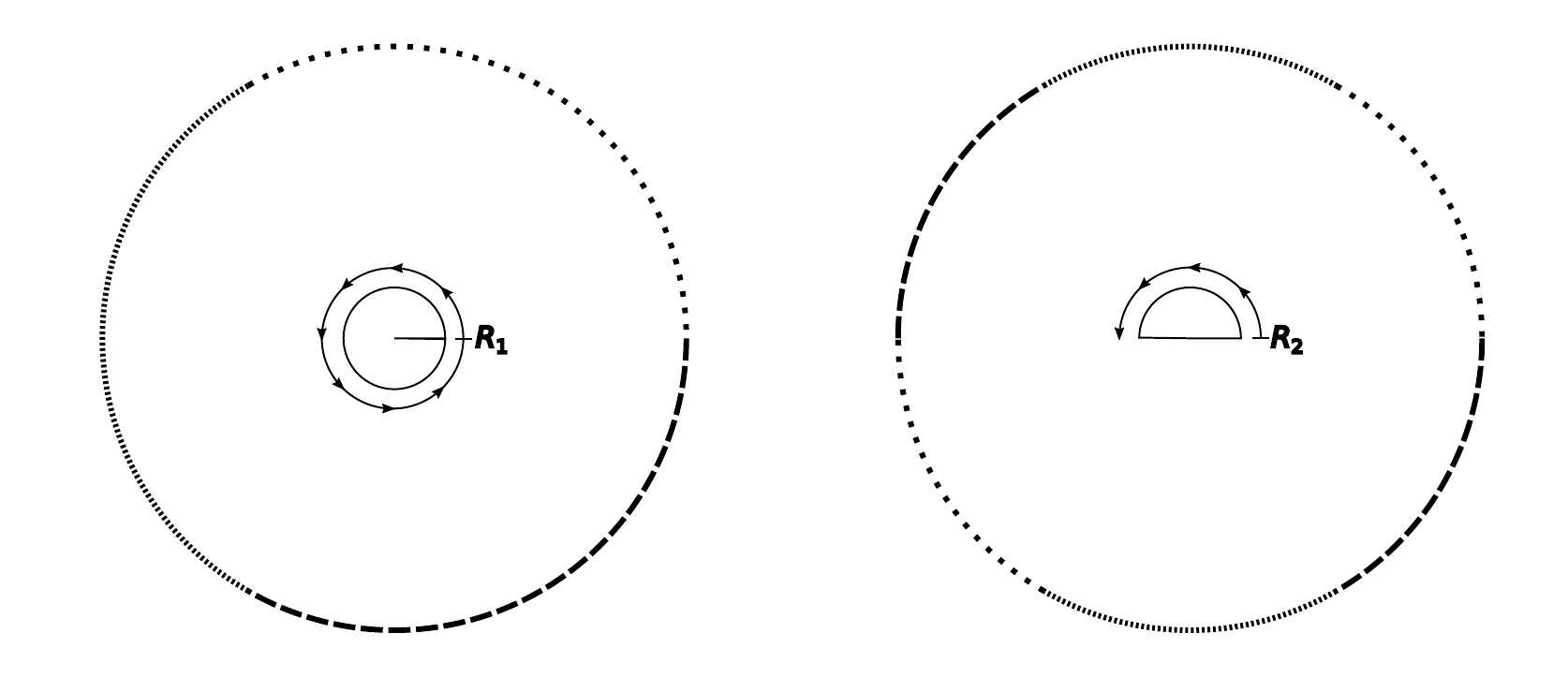}
		\caption{$C_1(P^{\oplus 3})$ and $C_2(P^{\oplus 3})$ along with the three embeddings of $C_1(P)$ and $C_2(P)$ (respectively).}
	\end{figure}
	
	\begin{rem}
		\label{rem:P_ints_Cn(P)}
		Remark~\ref{rem:dummy_constants} easily implies that every poset $P$ interprets each $C_n(P)$ without parameters.\xqed{\lozenge}
	\end{rem}
	
	\begin{rem}
		When $P$ is dense linear without endpoints, the structure $C_n(P)$ is essentially the $M_n$ described in Section~4 of \cite{CLPZ01}.\xqed{\lozenge}
	\end{rem}
%
%
%

	\begin{dfn}
		\label{dfn:wcts}
		A poset $(P,<)$ is \emph{cyclically three-splitting} if for every $n$ we have $C_n(P)\preceq C_n(P^{\oplus 3})$ and $C_n(P^{\oplus 2})\preceq C_n(P^{\oplus 3})$ (where the embeddings are induced by the initial embeddings of $P$ and $P^{\oplus 2}$ in $P^{\oplus 3}$).
		
		$P$ is \emph{weakly cyclically three-splitting} if there is a cyclically three-splitting poset $(Q,<)$ such that for all $n$ we have $C_n(P)\cong C_n(Q)$.\xqed{\lozenge}
	\end{dfn}
	
%
	\begin{rem}
		Note that the three copies of $C_n(P)$ in $C_n(P^{\oplus 3})$ are conjugate by automorphisms of $C_n(P^{\oplus 3})$, so if one of them is an elementary substructure, so are the other two.
		
		The same is true for the three copies of $C_n(P^{\oplus 2})$ in $C_n(P^{\oplus 3})$. \xqed{\lozenge}
	\end{rem}

	\begin{rem}
		\label{rem:tspl_to_ctspl}
		By Remark~\ref{rem:P_ints_Cn(P)}, an initially self-additive poset is cyclically three-splitting (because by Proposition~\ref{prop:self_add_propagates}, initial self-additivity implies that the initial embeddings of $P$ and $P^{\oplus 2}$ in $P^{\oplus 3}$ are elementary). The same is true for ``finally self-additive'' posets.\xqed{\lozenge}
	\end{rem}
	
	The prototypical examples of weakly cyclically three-splitting posets are the infinite discrete and dense linear orders.
	\begin{prop}
		\label{prop:lin_orders_are_wts}
		A dense or discrete linear order without endpoints is cyclically three-splitting.
		
		A discrete linear order with two endpoints is weakly cyclically three-splitting.
	\end{prop}
	\begin{proof}
		The first part is immediate by Proposition~\ref{prop:no_endpoints_three_splitting} and Remark~\ref{rem:tspl_to_ctspl}, as discrete and dense linear orders without endpoints are cyclically three-splitting by virtue of being initially self-additive.
		
		For the second part, note that an infinite discrete linear order with two endpoints is of the form $\omega\oplus L\oplus \omega^*$, where $L$ is discrete without endpoints (and $\omega^*$ is $\omega$ in reverse, i.e.\ an infinite descending chain).
		
		It is easy to see that for each $n$, $C_n(\omega\oplus L\oplus \omega^*)\cong C_n(L\oplus \bZ)$. Since $L\oplus \bZ$ is discrete without endpoints, it is cyclically three-splitting, so by definition $\omega\oplus L\oplus \omega^*$ is weakly cyclically three-splitting.
	\end{proof}
	
	\section{G-compactness of \texorpdfstring{$C_n(P)$}{Cn(P)} and linear orders}
	\label{sec:main_proof}
	\subsection*{Lascar diameters in \texorpdfstring{$C_n(P)$}{Cn(P)}}
	In this section, we fix a weakly cyclically three-splitting poset $P$, a natural number $n\geq 3$, and let $T_n:=\Th(C_n(P))$. The aim is to show that in a monster model of $T_n$, we can find a Lascar strong type of diameter at least $\lfloor n/2\rfloor$ (thus showing the lack of hereditary G-compactness of $P$). Thus we may (and do) assume without loss of generality that $P$ is cyclically three-splitting (because $T_n$ depends only on the isomorphism class of $C_n(P)$).
	
	We fix a monster model $\fC_n\succeq C_n(P^{\oplus 3})$ (by cyclical three-splitting, $\fC_n$ is a monster model of $T_n$). Denote by $S_n$ the definable binary relation given by $S_n(x,y)$ if $C(R_n^{-1}(x),y,R_n(x))$. It is helpful to think of $S_n(x,y)$ as saying that the distance between $x$ and $y$ is less than $1$.
	
	\begin{prop}
		\label{prop:props_of_Tn}
		$T_n$ implies the following:
		\begin{itemize}
			\item
			$\forall x\forall y\, \biglor_{i=1}^n S_n(x,R_n^i(y))$
			\item 
			$\forall x\forall y\, S_n(x,y)\iff S_n(y,x)$
			\item 
			For each $0\leq k<n/2$ we have $\forall x\, \neg S_n^k(x,R^{k}(x))$.
		\end{itemize}
	\end{prop}
	\begin{proof}
		It is enough to show that the statements are true in $C_n(P)$. Fix any $(p,i),(q,j)\in C_n(P)$.
		
		For the first statement, just note that $C(R_n^{-1}(p,i),R_n^{i-j}(q,j),R_n(p,i))$.
		
		For the second one, suppose $C(R_n^{-1}(p,i),(q,j),R_n(p,i))$. We need to show that $C(R_n^{-1}(q,j),(p,i),R_n(q,j))$ (the other implication is symmetric). If $i=j$, the conclusion is clear, so suppose $i\neq j$. Note that it implies that $\lvert i-j\rvert=1$ (or one of them is $1$, and the other is $n$). For simplicity, suppose $i\geq 2$, $j=i+1$ and $j<n$ (the other cases are similar) .
		
		Under those assumptions, $R_n^{-1}(p,i)=(p,i-1)$, $R_n(p,i)=(p,i+1)$ and $(q,j)=(q,i+1)$, so we have $C((p,i-1),(q,i+1),(p,i+1))$. By definition of $C$, this means that $q<p$, which clearly implies $C((q,i),(p,i),(q,i+2))$. Since we also have $R_n^{-1}(q,j)=(q,i)$ and $R_n(q,j)=(q,i+2)$, this means that $S_n((q,j),(p,i))$.
		
		For the third statement, we may assume without loss of generality that $i=1$. We need to show that $\neg S_n^k((p,1),R_n^k(p,1))$. Since $k<n/2$, $R_n^{k}(p,1)=(p,k+1)$. Suppose towards contradiction that we have $p=p_0,p_1,\ldots, p_k=p$ and $1=i_0,i_1,\ldots i_k=k+1$ such for each $j<k$ we have $S_n((p_j,i_j),(p_{j+1},i_{j+1}))$. Clearly, this implies that $d(i_j,i_{j+1})\leq 1$ (where $d$ is the cyclic distance). Since $d(i_0,i_k)=k$, by triangle inequality, we must have $d(i_j,i_{j+1})=1$, and since $k<n/2$, it follows that for each $j$ we must have $i_j=j+1$. On the other hand, it is easy to see that $S_n((p_j,j+1),(p_{j+1},j+2))$ implies that $C((p_j,j+1),(p_{j+1},j+2),(p_j,j+2))$, which holds only when $p_j>p_{j+1}$. But then by induction $p_0>p_k$, which contradicts the assumption that $p_k=p_0$.
%
%
	\end{proof}
	
	\begin{lem}
		\label{lem:basic_distance}
		For every $a,b\in \fC_n$, if $\Lascd(a,b)\leq 1$, then $S_n^2(a,b)$. More generally, if $\Lascd(a,b)\leq k$, then $S_n^{2k}(a,b)$.
	\end{lem}
	\begin{proof}
		For the first part, suppose that for some $c\in \fC_n$, $\tp(a/c)= \tp(b/c)$. By the preceding proposition, we have for some $i$ that $S_n(a,R_n^i(c))$. But then $S_n(x,R_n^i(c))\in \tp(a/c)$. It follows that $S_n(b,R_n^{i}(c))$, so by symmetry $S_n^2(a,b)$. In particular, if for some $M\preceq \fC_n$ we have $\tp(a/M)=\tp(b/M)$, then $S_n^2(a,b)$.
		
		The second part immediately follows: if $\Lascd(a,b)\leq k$, then we have a sequence $a_0=a,a_1,\ldots,a_n=b$ with $\Lascd(a_j,a_{j+1})\leq 1$, so by the first part, $S_n^2(a_j,a_{j+1})$, which clearly implies that $S_n^{2k}(a,b)$ holds.
	\end{proof}
	
	\begin{lem}
		\label{lem:main_cyclic_diameter}
		If $P$ is a weakly cyclically three-splitting poset, then for any $p_0\in C_n(P)$ and any $k<n/2$, we have that $\Lascd((p_0,1),R_n^k(p_0,1))> k/2$ and $(p_0,1)\equivLasc R_n^k(p_0,1)$.
	\end{lem}
	\begin{proof}
		By Proposition~\ref{prop:props_of_Tn}, we have $\neg S_n^k((p_0,1),R_n^k(p_0,1))$, and hence, by Lemma~\ref{lem:basic_distance}, $2\Lascd((p_0,1),R_n^k(p_0,1))>k$, so $\Lascd((p_0,1),R_n^k(p_0,1))>k/2$. It remains to show that $(p_0,1)\equivLasc R_n^k(p_0,1)$, i.e.\ that $\Lascd((p_0,1),R_n^k(p_0,1))$ is finite.
		
		Enumerate $P^{\oplus 3}$ naturally as $(p,j)_{p\in P,j\in \{1,2,3\}}$, and for each $j=1,2,3$ let $P_j=\{(p,j))\mid {p\in P}\}$, and let $P_{j,2}=\{(p,j'))\mid {p\in P\land j'\neq j}\}$. For brevity, we will use the convention that if $j=1$, then $j-1=3$ and if $j=3$, then $j+1=1$. We will also write simply $(p,j,i)$ for $((p,j),i)\in C_n(P^{\oplus 3})$. Note that for each $j$ we have $C_n(P_{j,2})=C_n(P_{j-1})\cup C_n(P_{j+2})$.
		
		By assumption, each $C_n(P_j)\preceq C_n(P^{\oplus 3})$. So we may identify $C_n(P)$ with $C_n(P_1)$ in such a way that (for each $p$) $(p,i)=(p,1,i)$. We will show that $\Lascd((p_0,1),R_n^k(p_0,1))=\Lascd((p_0,1,1),R_n^k(p_0,1,1))\leq 3k$.

		Let $\sigma$ be the automorphism of $C_n(P^{\oplus 3})$ given by
		\[
			\sigma(p,j,i)=\begin{cases}
			(p,j+1,i) & \textrm{if }j<3\\
			(p,1,i+1) &\textrm{if }j=3\textrm{ and }i<n \\
			(p,1,1) & \textrm{if }j=3\textrm{ and }i=n
			\end{cases}
		\]
		Note that $\sigma$ restricts to isomorphisms $C_n(P_1)\to C_n(P_2)$, $C_n(P_2)\to C_n(P_3)$ and $C_n(P_3)\to C_n(P_1)$, and also that $\sigma^3=R_n\restr_{C_n(P^{\oplus 3})}$.
		
		Now, for each $j=1,2,3$, we put a map $\sigma_j\colon C_n(P_{j,2})\to C_n(P_{j-1,2})$ given by identity on $C_n(P_{j-1})$ and by $\sigma$ on $C_n(P_{j+1})$. Note that each $\sigma_j$ is an isomorphism.
		
		Then, since $C_n(P_{j,2})$ and $C_n(P_{j-1,2})$ are elementary in $C_n(P^{\oplus 3})$, they are also elementary substructures of $\fC_n$. Hence, $\sigma_j$ are partial elementary maps in $\fC_n$ (as isomorphisms between elementary substructures).
		
		Thus, each $\sigma_j$ can be extended to an automorphism $\sigma_j^*\in \Aut(\fC_n)$, and clearly $\sigma_j^*\in \Aut(\fC_n/C_n(P_{j-1}))$. Now, notice that each $C_n(P_{j-1})\preceq \fC_n$, and furthermore, it is not hard to see that for all $p$ and $i$ we have
		\[
			\sigma^*_2\sigma^*_1\sigma^*_3(p,1,i)=\sigma_2\sigma_1\sigma_3(p,1,i)=\sigma^3(p,1,i)=R_n(p,1,i).
		\] It follows that $\Lascd((p_0,1,i),R_n(p_0,1,i))\leq 3$, and hence $\Lascd((p_0,1,1),R_n^k(p_0,1,1))\leq 3k$.
	\end{proof}
	
	\begin{thm}
		\label{thm:wts_not_hgc}
		If $P$ is a weakly cyclically three-splitting poset (cf.\ Definition~\ref{dfn:wcts}), then $P$ is not hereditarily G-compact.
	\end{thm}
	\begin{proof}
		Note that $P$ interprets the many-sorted structure $(C_n(P))_{n\in \bN}$. By Lemma~\ref{lem:main_cyclic_diameter}, we can find in each $C_n(P)$ a $\equivLasc$-class of diameter at least $\lfloor n/4\rfloor$. In particular, in $(C_n(P))_{n\in \bN}$ we have Lascar strong types of arbitrarily large Lascar diameter, so by Fact~\ref{fct:gcom_equivalent}, $(C_n(P))_{n\in \bN}$ is not G-compact, so $P$ is not hereditarily G-compact.
	\end{proof}
	
	\begin{rem}
		When $P$ is a dense linear ordering without endpoints, the many-sorted structure $(C_n(P))_{n\in \bN}$ is (up to elementary equivalence) exactly the structure given in \cite{CLPZ01} as an example of a structure with a non-G-compact theory.\xqed{\lozenge}
	\end{rem}
	\subsection*{Linear orders are not hereditarily G-compact}
	In this section, we will show that linear orders are not hereditarily G-compact (thus proving the main result of this paper, Theorem~\ref{thm:lin_not_hgc}). To that end, we will extract a dense or discrete linear order, and then apply Theorem~\ref{thm:wts_not_hgc} and Proposition~\ref{prop:lin_orders_are_wts}.
	
	Throughout the section, by an \emph{interval} in a linear order $K$ we mean a set of the form $(a,b)$, $[a,b]$, $[a,b)$ or $(a,b]$ for some $a,b\in K\cup \{-\infty,+\infty\}$, where $a\leq b$. A discrete linear order is one in which every point is isolated (in the order topology). A dense linear order is one in which every open interval (with distinct endpoints) is nonempty.
	
	Given a linear order $K$, denote by $P(K)$ the set of immediate predecessors in $K$ (i.e.\ elements such $a\in K$ such that for some $b\in K$, $b>a$, the interval $(a,b)$ is empty). Note that $K$ is dense if and only if $P(K)=\emptyset$.
	\begin{prop}
		\label{prop:succ_conv_comp}
		Let $(K,<)$ be an arbitrary linear order. If $I$ is a convex component of $P(K)$ in $K$ (i.e.\ a maximal subset of $P(K)$ which is convex in $K$), then either $I$ is finite or $I$ contains arbitrarily long finite intervals.
	\end{prop}
	\begin{proof}
		If $I$ is finite, we are done. Suppose, then, that $I$ is infinite.
		
		%
		Take any $a\in I$, and consider the set $S_a$ of all $b\in K$ such that either $b>a$ and $[a,b]$ if finite or $b<a$ and $[b,a]$ is finite (i.e.\ the set of all elements of $K$ that can be reached from $a$ by taking successors and predecessors). Note that $S_a$ is convex  in $K$ and all elements of $S_a$, except for the last one (if it exists) are contained in $I$.
		
		Thus, if $S_a$ is infinite, we are done (because we can find arbitrarily long finite intervals in $I\cap S_a$). So suppose $S_a$ is finite. Then it has a smallest element $a_-$ and a largest element $a^+$. Note that this implies that $a^+\notin P(K)$, so $I\setminus S_a=I\setminus [a_-,a^+]=I\cap (-\infty,a_-)$. Thus, because $I$ is infinite, there is some $b\in I$ such that $b<a_-$. Since (by convexity of $I$) we have $[b,a_-]\subseteq I\subseteq P(K)$, every element of $[b,a_-)$ has a successor. On the other hand, by definition of $a_-$, these successors are always strictly smaller than $a_-$. It follows that $[b,a_-)\subseteq I$ contains arbitrarily long finite intervals.
	\end{proof}
	Using Proposition~\ref{prop:succ_conv_comp}, we can perform the extraction mentioned before.
	\begin{lem}
		\label{lem:linear_reduction}
		If $(L,<)$ is an $\aleph_0$-saturated infinite linear order, then there is an infinite definable set $D\subseteq L$ such that $(D,{<}\restr_D)$ is dense without endpoints or discrete with two endpoints.
	\end{lem}
	\begin{proof}
%
		If $L$ contains arbitrarily long finite intervals, we can find (by $\aleph_0$-saturation) some $a<b$ such that $[a,b]$ is infinite and discrete, and then we are done. So let us assume that $L$ does not contain arbitrarily long finite intervals. We will show that $L':=L\setminus P(L)$ is an infinite dense linear ordering. (If $L'$ has any endpoints, we can just drop them to obtain a dense ordering without endpoints.)
		
		By the preceding paragraph, no convex component of $P(L)$ can contain arbitrarily long finite intervals, so by Proposition~\ref{prop:succ_conv_comp}, all convex components of $P(L)$ are finite.
		
		It follows that $L'$ is an infinite linear order (so in particular, it has at least two elements): indeed, if $L'$ was finite, its complement --- i.e. $P(L)$ --- would have finitely many convex components in $L$. But since we assume that all of them are finite, this would imply that $P(L)$ is finite, and hence so is $L=L'\cup P(L)$, a contradiction.
		
		Now, note that $P(L')=\emptyset$. Otherwise, if $a\in P(L')$, then there is some $b\in P(L')$ such that $a<b$ and $(a,b)\cap L'=\emptyset$. But $(a,b)\cap L'=(a,b)\setminus P(L)$. Thus, $(a,b)\cap P(L)=(a,b)$, so it is a subset of $P(L)$, convex in $L$, and hence finite. But then either $(a,b)=\emptyset$ --- in which case $a$ is the predecessor of $b$ in $L$ --- or there is a minimal element of $(a,b)$, and that element is a successor of $a$ in $L$. In both cases, $a\in P(L)$, which is a contradiction, as $a\in L'=L\setminus P(L)$.
		
		Now, since $P(L')=\emptyset$, it follows that $L'$ is dense, and we are done.
		%
	\end{proof}
	Note that the subset $D$ in the conclusion of Lemma~\ref{lem:linear_reduction} is weakly cyclically three-splitting (by Proposition~\ref{prop:lin_orders_are_wts}). The following remark shows that it is, in this way, the best possible result, as it is not hard to see that a three-splitting poset cannot have a maximum.
	\begin{rem}[by Antongiulio Fornasiero]
		\label{rem:discrete_two_ends}
		One can show that a discrete (pure) linear order with two endpoints does not interpret a linear order without a maximum (because it is definably compact, i.e.\ every uniformly definable family of definable sets with finite intersection property has nonempty intersection; this is preserved by interpretation and also clearly not true about an infinite linear order without a maximum).\xqed{\lozenge}
	\end{rem}
	
	Finally, we can prove the main theorem of this paper.
	\begin{thm}
		\label{thm:lin_not_hgc}
		If $T$ interprets an infinite linear order, then it is not hereditarily G-compact.
	\end{thm}
	\begin{proof}
		Without loss of generality, $T=\Th(L,<)$ for some $\aleph_0$-saturated infinite linear $(L,<)$. By Lemma~\ref{lem:linear_reduction}, we know that $L$ contains a definable subset which is either dense without endpoints or discrete with both endpoints. In both cases, by Proposition~\ref{prop:lin_orders_are_wts}, the induced order is weakly cyclically three-splitting, and so, by Theorem~\ref{thm:wts_not_hgc}, its theory is not hereditarily G-compact, and neither is $T$.
	\end{proof}
	
	Recall the following long-standing conjecture about NIP theories.
	\begin{con}
		\label{con:unstable_nip_linear}
		If $T$ is an unstable NIP theory, then $T$ interprets an infinite linear order.\xqed{\lozenge}
	\end{con}
	
	It is known that the conjecture holds for $\aleph_0$-categorical structures. For arbitrary unstable NIP theories, it is known that there exists a $\biglor$-definable linear quasi-ordering with an infinite chain (on some definable set); see \cite[Theorems 5.12, 5.13]{Sim19}.
	
	(In, \cite{GL13}, the authors showed also that an unstable theory interprets an infinite linear order under a rather strong assumption of weak VC-minimality.)
%
	
	Using the conjecture, we can formulate the following conditional corollary.
	\begin{cor}
		\label{cor:NIP+SOP_nhgc}
		If Conjecture~\ref{con:unstable_nip_linear} holds, then every theory which is both NIP and SOP is not hereditarily G-compact.
	\end{cor}
	\begin{proof}
		Immediate by Theorem~\ref{thm:lin_not_hgc}.
	\end{proof}
	As stated before, Conjecutre~\ref{con:unstable_nip_linear} actually holds for $\aleph_0$-categorical theories, so we can say the following unconditionally.
	\begin{cor}
		An $\aleph_0$-categorical NIP theory is stable if and only if it is hereditarily G-compact.
	\end{cor}
	\begin{proof}
		In one direction, this follows from Fact~\ref{fct:stable_simple_gc}. In the other direction, by \cite[Theorem 5.12]{Sim19}, we have a definable set $X$ and a dense linear quasi-order $R$, $\bigvee$-definable over a finite set $A$. By $\aleph_0$-categoricity, it follows that $R$ is actually definable, which allows us to interpret an infinite linear order. The conclusion follows by Theorem~\ref{thm:lin_not_hgc}.
	\end{proof}
	
	It might be interesting to see if the results of \cite{Sim19} can be used to prove Corollary~\ref{cor:NIP+SOP_nhgc} unconditionally. \cite[Theorem 5.12]{Sim19} does give us dense linear quasi-order, the main issue is that since the structure is not just a pure order, there seems to be no obvious reason for it to be self-additive (with the whole structure necessary to make the quasi-order $\biglor$-definable), even if there are no endpoints. If one could find some sort of canonical reduct or definable subset with that property (or at least the analogue of weak cyclic three-splitting), then it seems like the proof of Theorem~\ref{thm:wts_not_hgc} could be adapted to show the lack of hereditary G-compactness.

	\section{Criteria for (hereditary) G-compactness}
	\label{sec:criteria}
	
	\subsection*{G-compactness is finitary}
	In this section, we show that in some ways, G-compactness is finitary. This might be helpful in proving hereditary G-compactness.
	
	\begin{rem}
		If we have two languages $L_0\subseteq L$ and a monster model $\fC$ for both $L_0$ and $L$, then for any $a,b\in \fC$ we have trivially that $\Lascd(a,b)\leq n$ in $\fC|_L$ implies the same in the sense of $\fC|_{L_0}$. Consequently, any $L_0$-formula implied by $\Lascd(x,y)\leq n$ in the sense of $\fC|_L$ is also implied by $\Lascd(x,y)\leq n$ in the sense of $\fC|_{L_0}$.
	\end{rem}
	
	\begin{prop}
		\label{prop:finite_sublanguage}
		Let $T$ be a first order theory in language $L$. If for every finite $L_0\subseteq L$ we have that $T\restr_{L_0}$ is $n$-G-compact, then $T$ is $n$-G-compact.
	\end{prop}
	\begin{proof}
		The proof is by contraposition. Suppose $T$ is not $n$-G-compact. This means that we have two tuples $a,b\in \fC$ such that $\Lascd(a,b)\leq n+1$ but $\Lascd(a,b)>n$, i.e.\ $\neg \Lascd(a,b)\leq n$. Since $\Lascd(a,b)\leq n$ is an $\emptyset$-type-definable condition, by compactness, there is a formula $\varphi(x,y)$ witnessing its failure, i.e.\ such that $\Lascd(x,y)\leq n\vdash \varphi(x,y)$ and $\not\models \varphi(a,b)$. Now, $T\cup \{\neg\varphi(a,b),\Lascd(a,b)\leq n+1\}$ is consistent. It follows that if we take $L_0$ to be the finite set of symbols used in $\varphi$, then so is $T\restr_{L_0}\cup \{\neg\varphi(a,b),\Lascd(a,b)\leq n+1\}$ (where $\Lascd$ is in the sense of $T\restr_{L_0}$), so $T\restr_{L_0}$ is not $n$-G-compact.
	\end{proof}
	
	\begin{cor}
		\label{cor:finit_params}
		To show hereditary G-compactness of a theory, it is enough to show that for some $n$, all theories interpreted with finitely many parameters are $n$-G-compact. Furthermore, it is enough to choose the parameters from a single model realising all finitary types over $\emptyset$ (e.g.\ an $\aleph_0$-saturated model). Likewise, to check weakly hereditary G-compactness, it is enough to find an uniform bound on $n$-G-compactness of reducts of expansions of $T$ by finitely many constants from such a model.
	\end{cor}
	\begin{proof}
		This is immediate by Proposition~\ref{prop:finite_sublanguage} and Fact~\ref{fct:gcom_equivalent}.
	\end{proof}

	\begin{rem}
		\label{rem:heredit_ngc_equiv}
		One can also show that a theory $T$ is hereditarily G-compact if and only if if for some $n$, it is hereditarily $n$-G-compact, so the sufficient condition for hereditary G-compactness given in Corollary~\ref{cor:finit_params} is also necessary. (But there seems to be no obvious reason for this to be true for weakly hereditary G-compactness.)\xqed{\lozenge}
	\end{rem}

	\subsection*{Stable and simple theories are G-compact}
	Simple theories form the widest class of theories known to be hereditarily G-compact.
	\begin{fct}
		\label{fct:stable_simple_gc}
		Every simple theory is $2$-G-compact, and hence hereditarily $2$-G-compact. Every stable theory is $1$-G-compact, and hence hereditarily $1$-G-compact.
	\end{fct}
	\begin{proof}
		Simple theories have existence (see Definition~\ref{dfn:existence}) and are both NTP$_2$ and NSOP$_1$, so G-compactness follows from either one of Fact~\ref{fct:gcomp_nsop1} and Fact~\ref{fct:gc_in_ntp2}. (Classically, this follows from the independence theorem for simple theories.) Heredity follows from the fact that simplicity is hereditary.
		
		For stable theories, it is immediate by the fact that types over $\acl^{\eq}(\emptyset)$ are stationary: if $a\equivLasc b$, then trivially $a\equiv_{\acl^{\eq}(\emptyset)} b$, so for any $M\forkindep ab$ we have $a\equiv_M b$.
	\end{proof}
	
	\subsection*{Fraïssé limits are (often) G-compact}
	Fraïssé limits are frequently used to construct examples in model theory, and often, they are limits of Fraïssé classes of relational structures with free amalgamation (also over $\emptyset$). The property of having canonical JEP is slightly more general, and we will see that it is sufficient for G-compactness.
	\begin{dfn}
%
		Let $\cC$ be a Fraïssé class in an arbitrary language. We say that $\cC$ has \emph{canonical JEP} (or \emph{canonical amalgamation over $\emptyset$}) if we have a functor from $\cC^2$, taking $(A,B)$ to a cospan of the form $A\rightarrow A\otimes B\leftarrow B$. In other words, given a pair $A,B\in \cC$, we have an amalgam $A\otimes B$ which respects embeddings, in the sense that given a pair $i_{AB},i_{CD}$ of embeddings, we have a canonical embedding $i_{AB}\otimes i_{CD}$ such that the following diagram commutes.
		\begin{center}
			\begin{tikzcd}
			B\ar[r] & B\otimes D & D\ar[l]& \\
			A\ar[u,"i_{AB}"]\ar[r] & A\ar[u,"i_{AB}\otimes i_{CD}"]\otimes C & C\ar[u,"i_{CD}"]\ar[l]&\xqed{\lozenge}
			\end{tikzcd}
		\end{center}
	\end{dfn}
	\begin{rem}
		Note that if we have free JEP, we can take $\otimes$ to be simply the disjoint union for structures as well as embeddings.)\xqed{\lozenge}
	\end{rem}
	
	\begin{prop}
		\label{prop:fraisse_gcomp}
		Let $\cC$ be a Fraïssé class with free amalgamation, or more generally, with canonical amalgamation. Then the theory of the Fraïssé limit of $\cC$ is $1$-G-compact.
	\end{prop}
	\begin{proof}
		The proposition follows easily from the following Claim --- namely, by compactness and Claim, given two tuples of the same type, we can build a model over which they are equivalent.
		\begin{clm*}
			Let $M$ be the limit of $\cC$, and let $A,B\subseteq M$ be finitely generated substructures. Then given any $A'\subseteq M$, $A'\cong A$, we can find some $B'\subseteq M$, $B'\cong B$, such that $\tp(A/B')=\tp(A'/B')$.
		\end{clm*}
		\begin{clmproof}
			Let $C\subseteq M$ the substructure generated by $A$ and $A'$, and consider $C\otimes B$ as a substructure of $M$ in such a way that the canonical embedding of $C$ is simply the inclusion. We claim that the canonical copy $B'$ of $B$ in $C\otimes B$ satisfies the conclusion. To see this, consider the following commutative diagram.
			\begin{center}
				\begin{tikzcd}
				C\ar[rr] & & C\otimes B && C\ar[ll]\\
				&& B\ar[u]\ar[dl]\ar[dr] &&\\
				A\ar[r]\ar[uu,"\subseteq"] & A\otimes B\ar[rr,"\cong\otimes \id_B"]\ar[ruu,"\subseteq\otimes \id_B"] &&A'\otimes B\ar[ll]\ar[luu,"\subseteq\otimes \id_B",swap] & A'\ar[l]\ar[uu,"\subseteq"]
				\end{tikzcd}
			\end{center}
			In this diagram, the unmarked arrows are the canonical embeddings. Commutativity follows from the definition of canonical JEP, and it easily implies the conclusion --- the type of $A$ over $B'$ is coded by the embedding of $A$ into $A\otimes B$, which is isomorphic to the embedding into $A'\otimes B$ induced by the isomorphism $A\cong A'$.
		\end{clmproof}
	\end{proof}
	
	\begin{ex}
		\label{ex:canonical_jep}
		Proposition~\ref{prop:fraisse_gcomp} easily implies that many theories are $1$-G-compact, including essentially all well-understood NSOP+SOP$_1$ theories. Among them are (the theories of) the following:
		\begin{itemize}
			\item
			various generic graphs, such as:
			\begin{itemize}
				\item
				the (Rado) random graph,
				\item 
				the (Henson) generic $K_n$-free graphs (cf.\ Example~\ref{ex:henson}),
				\item
				directed graphs omitting odd $({\leq} n)$-cycles,
			\end{itemize}
			\item 
			vector spaces with generic bilinear forms,
			\item 
			atomless Boolean algebras,
			\item
			free commutative monoids (e.g.\ $(\bN\setminus\{0\},\cdot)$),
			\item
			dense linear orderings without endpoints.\xqed{\lozenge}
		\end{itemize}
	\end{ex}
%
%

	\subsection*{G-compactness of NSOP\texorpdfstring{$_1$}{1} theories}
	Recall the following definition.
	\begin{dfn}
		\label{dfn:existence}
		Let $T$ be a first order theory with monster model $\fC$. Given a small $A\subseteq \fC$, we say that \emph{$T$ has existence (for forking) over $A$} or that \emph{$A$ is a base for forking} if it satisfies one of the following (equivalent) conditions:
		\begin{itemize}
			\item
			every type over $A$ can be extended to a complete global type, nonforking over $A$,
			\item
			for any $a,b$, there is some $a'\equiv_A a$ such that $a'\forkindep_Ab$,
			\item
			no consistent formula with parameters in $A$ forks over $A$.
		\end{itemize}
		We say that \emph{$T$ has existence (for forking)} if it has existence over every small set.\xqed{\lozenge}
	\end{dfn}
	
	Using the results of \cite{DKR19} one can show that in a theory with NSOP$_1$ (which is a property strictly stronger than strict order property and strictly weaker than simplicity), the so-called Kim-independence relation $\forkindep^K$ (which refines the forking independence relation) has properties which allow us to prove $2$-G-compactness. The following fact summarises all these properties --- the reader can take $\forkindep^K$ as an abstract ternary relation satisfying them. For more in-depth explanations, see \cite{KR19} and \cite{DKR19}.
	\begin{fct}
		\label{fct:nsop1}
		Suppose $T$ is an NSOP$_1$ theory with existence. Then there is a ternary independence relation $\forkindep^K$ refining the forking independence relation, satisfying the following axioms:
		\begin{enumerate}
			\item 
			symmetry: $a\forkindep^K_Cb$ if and only if $b\forkindep^K_Ca$,
			\item 
			existence: for any $a,b,C$, there is some $a'\equiv_C a$ such that $a'\forkindep^K_Cb$,
			\item 
			transitivity: if $C\subseteq D$ and $a\forkindep^K_C b$ and $a\forkindep^K_DC$, then $a\forkindep^K_D b$,
			\item
			invariance under automorphisms: if $abC\equiv a'b'C'$, then $a\forkindep_C^Kb$ if and only if $a'\forkindep_{C'}^K b'$.
		\end{enumerate}
		In addition, $\forkindep^K$ satisfies the following ``3-amalgamation'' theorem.
		
		If $b\forkindep^K c$, $a\forkindep^K b$, $a'\forkindep^K c$ and $a\equivLasc a'$, then there is some $a''$ such that $a\equiv_b a''\equiv_c a'$.
	\end{fct}
	\begin{proof}
		Symmetry is \cite[Corollary 4.9]{DKR19}. Existence follows from the fact that ${\forkindep^K}\subseteq {\forkindep}$ (and existence for $\forkindep$). Transitivity is not included in \cite{DKR19}, but we have been informed (in personal communication with the third author) that it holds. Invariance under automorphisms is immediate by definition of $\forkindep^K$.
		
		The last paragraph is a part of the 3-amalgamation theorem for Lascar strong types (\cite[Theorem 5.8]{DKR19}).
	\end{proof}
	
	\begin{fct}
		\label{fct:gcomp_nsop1}
		Suppose $T$ is NSOP$_1$ with existence. Then $T$ is $2$-G-compact.
	\end{fct}
	\begin{proof}
		Fix any $a\equivLasc a'$. Take any $M_1'\preceq \fC$. By existence, we can find some $M_2'\equiv M_1'$ such that $M_2'\forkindep^K M_1'$. Again by existence, there is some $b\equiv_a a'$ such that $b\forkindep_a^K M_1'M_2'$, and thus (by applying an automorphism taking $ab$ to $aa'$) we can find some models $M_1,M_2$ such that $a'\forkindep^K_a M_1M_2$, $a\forkindep^K M_1M_2$ and $M_1\forkindep^K M_2$.
		
		By transitivity and symmetry, it follows that $aa'\forkindep^K M_1M_2$, so in particular $a\forkindep^K M_1$ and $a'\forkindep^K M_2$. By the last part of Fact~\ref{fct:nsop1}, we conclude that there is some $a''$ such that $a\equiv_{M_1} a''\equiv_{M_2} a'$, and hence $\Lascd(a,a')\leq 2$.
	\end{proof}
%
%
	\begin{rem}
		\label{rem:existence_for_nsop1}
		Note that whether or not every NSOP$_1$ theory has existence for forking is an open question. If the answer is positive, this would imply that NSOP$_1$ theories are hereditarily $2$-G-compact. Therefore, it seems unlikely that there would be a non-hereditarily G-compact NSOP$_1$ theory.\xqed{\lozenge}
	\end{rem}

	\subsection*{G-compactness of NTP\texorpdfstring{$_2$}{2} theories}
	
	\begin{fct}
		\label{fct:ntp2_bc}
		Suppose that $T$ is an NTP$_2$ theory with existence over $\emptyset$. Let $a,b,b',c$ be small tuples such that $c\forkindep ab$, $a\forkindep bb'$ and $b\equivLasc b'$. Then there is some $c'$ such that $c'a\equiv ca$ and $c'b'\equiv cb$.
	\end{fct}
	\begin{proof}
		This is \cite[Theorem 3.3]{BC14} (with $A=\emptyset$).
	\end{proof}
	
	\begin{fct}
		\label{fct:gc_in_ntp2}
		If $T$ is an NTP$_2$ theory with existence over $\emptyset$, then $T$ is $2$-G-compact.
	\end{fct}
	\begin{proof}
		Take any $b\equivLasc b'$. By existence, we can find models $M_a,M_c$, enumerated by $m_a,m_c$ (respectively), such that $m_a\forkindep bb'$ and $m_c\forkindep m_ab$. By Fact~\ref{fct:ntp2_bc} (with $a=m_a$ and $c=m_c$), there is a tuple $m_c'$ enumerating a model $M_c'$ such that $m_c'\equiv_{M_a} m_c$ and $m_c'b'\equiv m_cb.$ Now, let $b''$ be such that $m_c'b''\equiv_{M_a} m_cb$. Then also $m_c'b''\equiv m_c b\equiv m_c'b'$, so $b\equiv_{M_a} b''\equiv_{M_c'} b'$, whence $\Lascd(b,b')\leq 2$.
	\end{proof}
	
	An interesting subclass of NTP$_2$ theories consists of so-called inp-minimal theories. The inp-minimal theories which have NIP are exactly the dp-minimal theories, including all weakly o-minimal theories.
	
	The following fact shows that in this context, the existence of an appropriate linear order yields G-compactness. This is in stark contrast to Theorem~\ref{thm:lin_not_hgc}, which shows that the presence of a linear order allows us to destroy G-compactness.
	\begin{fct}
		\label{fct:inpm_gcomp}
		An inp-minimal linearly ordered theory (in the sense that there is a definable order on the universe, which is linear without endpoints) has existence. Consequently, every inp-minimal theory is $2$-G-compact. (This includes in particular all order dp-minimal, and hence all o-minimal theories.)
	\end{fct}
	\begin{proof}
		The proof of \cite[Corollary 2.6]{Sim11} is essentially a proof of the first sentence. The ``consequently'' part follows from the fact that inp-minimal theories have NTP$_2$ and Fact~\ref{fct:gc_in_ntp2}.
	\end{proof}
	
	\subsection*{Amenable theories}
	In \cite{HKP19}, the authors introduced the notion of \emph{first order amenability} of a first order theory, which is another generalisation of stability. They also show the following.
	\begin{fct}
		\label{fct:amenable_gcomp}
		If the theory $T$ is amenable, then it is G-compact.
	\end{fct}
	\begin{proof}
		This is \cite[Theorem 4.30]{HKP19}.
	\end{proof}
	
	Amenability is known not to be preserved by adding parameters, so Fact~\ref{fct:amenable_gcomp} does give us any new examples of (even weakly) hereditarily G-compact theories.
	
	\begin{rem}
		One can show that the limit of a Fraïssé class with canonical amalgamation is amenable (see the introduction to \cite{HKP19}), so Fact~\ref{fct:amenable_gcomp} implies Proposition~\ref{prop:fraisse_gcomp}.\xqed{\lozenge}
	\end{rem}

	\section{Open problems}
	\label{sec:open}
	We know that simple theories are G-compact (and hence hereditarily G-compact). Furthermore, Fact~\ref{fct:gcomp_nsop1} (along with the fact that the question of whether existence holds in general for NSOP$_1$ theories remains open) suggests that the same may well be true for NSOP$_1$ theories. On the other hand, Theorems~\ref{thm:lin_not_hgc} and \ref{thm:wts_not_hgc} suggests that strict order property may be an essential obstruction to hereditary G-compactness --- even though SOP yields only a poset with an infinite chains, not a linear order, it might be possible to use Theorem~\ref{thm:wts_not_hgc} (or some variant) in the nonlinear case. Taken together, this suggests that there may be a strong relationship between NSOP and hereditary G-compactness, which prompts the following two questions.
	
	\begin{qu}
		\label{qu:SOP_nhgc}
		Is every SOP theory not hereditarily G-compact?\xqed{\lozenge}
	\end{qu}
	
	\begin{qu}
		\label{qu:nsop_gcomp}
		Is every NSOP$_1$ theory G-compact? More generally, is every NSOP theory G-compact (equivalently, hereditarily G-compact)?\xqed{\lozenge}
	\end{qu}
	
	So far, there seems to be no compelling evidence in either direction for NSOP+SOP$_1$ theories. In particular, it seems to be completely open whether or not all NSOP$_4$ theories are G-compact (which would also make them hereditarily G-compact).
	
	Note that Example~\ref{ex:canonical_jep} can be extended to include large part of the known non-simple NSOP theories, ruling them out as possible witnesses to a negative answer to Question~\ref{qu:nsop_gcomp}.
	
	In an altogether different direction, it seems that whenever G-compactness is proved in any broad context, we actually have $2$-G-compactness. This suggests that the following question may be reasonable.
	
	\begin{qu}
		Suppose $T$ is hereditarily G-compact. Is $T$ necessarily $2$-G-compact? If not, is there such $n$ such that all hereditarily G-compact theories are $n$-G-compact?
	\end{qu}
	
	Resolving such general questions can be very difficult; on the other hand, it may be interesting to show hereditary G-compactness for \emph{any} SOP$_1$ theory. Unfortunately, analysing all reducts of $T^\eq$ does not seem to be easy, either, even for relatively well-understood $T$. Thus, it may be interesting to show at least \emph{weakly} hereditary G-compactness. A particularly promising class of examples to check are the $\aleph_0$-categorical theories with quantifier elimination in a finite relational language --- a famous conjecture of Simon Thomas says that they have finitely many reducts (without parameters), so it is plausible that we could simply classify (at least broadly) all the reducts (even with parameters) of a given theory and show that they are all G-compact.
	
	\begin{ex}
		\label{ex:henson}
		Fix some integer $n\geq 3$. Denote by $K_n$ the complete graph on $n$ vertices, and let $(H_n,E)$ be the generic $K_n$-free graph (the Fraïssé limit of the class of finite $K_n$-free graphs). One can show the following:
		\begin{itemize}
			\item
			the theory of $H_n$ is $1$-G-compact, TP$_2$, NSOP$_4$ and SOP$_3$ ($1$-G-compactness follows from Proposition~\ref{prop:fraisse_gcomp}),
			\item
			the only proper reduct of $(H_n,E)$ (without parameters) is the pure set,
			\item
			given any vertex $0\in H_n$, $(H_n,E,0)$ has at most sixteen reducts (without additional parameters), all of which are $1$-G-compact (see \cite[Theorem 2.3]{Pon17} for a description of all the reducts; the G-compactness follows from the observation that, roughly speaking, the proper reducts are (up to interdefinability) isomorphic to disjoint unions of $H_n$, $H_{n-1}$, pure sets and $\{0\}$).\xqed{\lozenge}
		\end{itemize}
	\end{ex}
	
	\begin{qu}
		\label{qu:henson}
		Is the theory of $H_n$ hereditarily G-compact (or at least weakly hereditarily G-compact) for any $n\geq 3$?\xqed{\lozenge}
	\end{qu}
	
	Finally, it might be interesting to see whether there are SOP, or even linearly ordered theories which are weakly hereditarily G-compact. We have seen that the latter are not hereditarily G-compact, and it is not hard to see that e.g.\ divisible ordered abelian groups are not weakly hereditarily G-compact, but the same question for pure linear orders seems less clear.
	
	We finish with a question about some possibly interesting candidates for weakly hereditary G-compactness.
	\begin{qu}
		Are the following weakly hereditarily G-compact?
		\begin{itemize}
			\item
			the theory of dense linear orderings without endpoints,
			\item
			the theory of atomless Boolean algebras.\xqed{\lozenge}
		\end{itemize}
	\end{qu}
	
	\section*{Acknowledgements}
	I would like to thank Krzysztof Krupiński, Slavko Moconja, Itay Kaplan and Antongiulio Fornasiero for helpful discussions. I am also grateful to Nicholas Ramsey for explaining some aspects of Kim-independence in NSOP$_1$ theories.
	
	\printbibliography

\end{document}